\documentclass[a4paper,english,11pt]{article}

\usepackage[german, english]{babel}
\usepackage[pdftex]{color, graphicx}
\usepackage{dsfont}
\usepackage{amsmath}
\usepackage{amssymb}
\usepackage[normalem]{ulem}
\usepackage[utf8]{inputenc}
\usepackage{enumerate}
\usepackage{ifthen}
\usepackage{paralist}
\usepackage[]{nicefrac}
\usepackage{fancyhdr}
\usepackage[style=american]{csquotes}

\newtheorem{theorem}{Theorem}[section]
\newtheorem{corollary}[theorem]{Corollary}
\newtheorem{proposition}[theorem]{Proposition}
\newtheorem{lemma}[theorem]{Lemma}
\newtheorem{remark}[theorem]{Remark}
\newtheorem{definition}[theorem]{Definition}
\newtheorem{claim}[theorem]{Claim}
\newtheorem{problem}[theorem]{Problem}
\newtheorem{notation}[theorem]{Notation}
\newtheorem{beispiel}[theorem]{Example}

\newtheorem{conjecture}[theorem]{Conjecture}

\newenvironment{theo}[1][\empty]{\begin{theorem} 
\ifthenelse{\equal{#1}{\empty}} {}  {\itshape(#1)} \normalfont ~\\}{\end{theorem}}
\newenvironment{cor}[1][\empty]{\begin{corollary} 
\ifthenelse{\equal{#1}{\empty}} {}  {\itshape(#1)} \normalfont ~\\}{\end{corollary}}
\newenvironment{prop}[1][\empty]{\begin{proposition} 
\ifthenelse{\equal{#1}{\empty}} {}  {\itshape(#1)} \normalfont ~\\}{\end{proposition}}
\newenvironment{lem}[1][\empty]{\begin{lemma} 
\ifthenelse{\equal{#1}{\empty}} {}  {\itshape(#1)} \normalfont ~\\}{\end{lemma}}
\newenvironment{rem}[1][\empty]{\begin{remark} 
\ifthenelse{\equal{#1}{\empty}} {}  {\itshape(#1)} \normalfont ~\\}{\end{remark}}
\newenvironment{defi}[1][\empty]{\begin{definition} 
\ifthenelse{\equal{#1}{\empty}} {}  {\itshape(#1)} \normalfont ~\\}{\end{definition}}

\newenvironment{proof}[1][\empty]{\ifthenelse{\equal{#1}{\empty}} {\paragraph{\textbf{Proof.}}~\\}  
{\paragraph{\textbf{Proof} of #1.}~\\}} {\hfill $\Box$}

\newcommand {\N}[0] {\mathbb{N}}
\newcommand {\R}[0] {\mathbb{R}}
\renewcommand {\S}[0] {\mathbb{S}}
\newcommand {\B}[0] {\mathbb{B}}

\newcommand {\NP}[0] {\mathbb{NP}}

\newcommand{\conv}[0] {\mathrm{conv}}
\newcommand{\aff}[0] {\mathrm{aff}}
\newcommand{\lin}[0] {\mathrm{lin}}

\newcommand{\ext}[0] {\mathrm{ext}}

\newcommand{\bd}[0]{\mathrm{bd}}
\renewcommand{\int}[0]{\mathrm{int}}
\newcommand{\relint}[0]{\mathrm{relint}}

\newcommand{\CC}[0]{\mathcal{C}}

\newcommand{\CH}[0]{\mathcal{H}}

\newcommand{\CV}[0]{\mathcal{V}}

\newcommand{\st}[0]{such that}

\begin{document}
\begin{titlepage}
  \renewcommand{\thefootnote} {\fnsymbol{footnote}} \vspace*{1,0cm}
  \begin{center}
 \LARGE  {\bf Sharpening Geometric Inequalities\\ using Computable Symmetry Measures} 
 \\[25mm]
 \normalsize
 \begin{minipage}[c]{10pt}
   \mbox{}
 \end{minipage}
 \begin{minipage}[c]{170pt}
   Ren\'{e} Brandenberg \\
   ~\\
   Zentrum Mathematik\\
   Technische Universität München\\
   Boltzmannstr.~3\\
   85747 Garching bei München \\
   Germany \\ 
   ~\\
   E-mail: brandenb@ma.tum.de
 \end{minipage} 
 \hfill
 \begin{minipage}[c]{10pt}
   \mbox{}
 \end{minipage}
\begin{minipage}[c]{170pt}

  Stefan K\"{o}nig \\
  ~\\
  Institut für Mathematik\\
  Technische Universität Hamburg-Harburg\\
  Schwarzenbergstr. 95\\
  21073 Hamburg\\
  Germany \\
  ~\\
  E-mail: stefan.koenig@tuhh.de
 \end{minipage}
 
 \vspace*{2.5cm}

\begin{minipage}[c]{410pt}
\noindent {\sc Abstract.}    
Many classical geometric inequalities on functionals of convex bodies depend on the dimension of the ambient space. We show that this dimension dependence may  often be replaced (totally or partially) by different symmetry measures of the convex body. Since these coefficients are bounded by the dimension but possibly smaller, our inequalities sharpen the original ones. Since they can often be computed efficiently, the improved bounds may also be used to obtain better bounds in approximation algorithms.
\end{minipage}

\vspace*{1cm}

\begin{minipage}[c]{410pt}
\end{minipage}

\vspace*{1cm}
\begin{minipage}[c]{410pt}
This is a preprint. The proper publication in final form is available at journals.cambridge.org, DOI 10.1112/S0025579314000291.
\end{minipage}
\end{center}

\renewcommand{\thefootnote}%
{\arabic{}}

\end{titlepage}

\selectlanguage{english}

\renewcommand{\subseteq}{\subset}
\renewcommand{\supseteq}{\supset}
\newcommand{\rec}{\mathrm{rec}}
\newcommand{\ls}{\mathrm{ls}}

\section{Introduction}
\paragraph{} Since Jung's famous inequality \cite{jung-01} in 1901, geometric inequalities relating different radii of convex bodies form a central area of research in convex geometry. Starting with \cite{bonnesenFenchelT}, in many classic works of convexity, significant parts are devoted to geometric inequalities among radii (e.g.~\cite{bottemaGeomIneq}, \cite[Section 6]{dgk-63}, \cite[Chapter 6]{egglestonConvexity}, \cite[Section 4.1.3]{hadwigerBuch}). 

Interesting and beautiful results of their own, geometric inequalities also serve as indispensable tools for many results in convex geometry itself as well as in other application areas. It is therefore not surprising that results such as Jung's Inequality \cite{jung-01} or John's Theorem \cite{john} still are frequently cited in a broad variety of papers (see e.g.~\cite{henkJohn} on L\"owner-John ellipsoids). Thus, even more than a century after Jung's seminal inequality, the area of geometric inequalities in general and especially among radii is still a prosperous field of research (see \cite{boltyanski-martini-06, isoradials, bernardo-ratioRadii, henk92, hernandezIneqNd, perelman, scott} for inequalities among radii of convex bodies and \cite{HenkBetke-RadiiVolumes, hernandez,  HenkHernandez-RadiiVolumes} for inequalities involving radii and other geometric functionals).

\paragraph{} The kind of inequalities to be considered in the following usually bound a geometric functional (e.g.~a certain radius) of a convex body in terms of another one. The statement of the theorem is then usually in two parts: a general bound on the ratio of these two functionals 
that holds true for arbitrary convex bodies and an additional statement that the bound can be improved (sometimes to a trivial bound) if the body under investigation is symmetric. In this paper, we propose to use measures of symmetry to sharpen geometric inequalities for convex bodies that are not symmetric but possibly far from the worst case bound in the original theorem. We also refer to \cite{bf-08, kaijser-guo-99, kaijser-guo-2003} for related work in the same lines and especially to \cite{br-07, schneider-09}, demonstrating already the basic idea of the approach which we follow here.

In particular, we prove sharpened versions of a classic inequality between in- and circumradius (e.g.~\cite[p. 28]{fejesToth1953lagerungen}), and of the famous theorems of  \emph{Jung} \cite{jung-01}, \emph{Steinhagen} \cite{steinhagen}, \emph{Bohnenblust} \cite{bohnenblust-38}, and \emph{Leichtweiß} \cite{leichtweiss}. Moreover, we present a compact proof of an improved version of \emph{John's} inequality.

\paragraph{} The symmetry measures that we use for this purpose are variants of Minkowski's measure of symmetry and have the desirable advantage that they are computable for polytopes via Linear Programming (see Lemmas~\ref{lem:computeSC} and \ref{lem:computeS0}). Hence, the improvement from basing these inequalities on symmetry coefficients is not only of theoretical interest but also allows better bounds in practical applications and in particular in core set algorithms (see e.g.~\cite{coresetPaperDCG}).

As a noteworthy remark, our inequalities show, that in many cases the ratio between two functionals is bound solely to the symmetry coefficients and does not intrinsically depend on the dimension. The dimension dependence, which is known from the original theorems, only enters
the inequalities as a worst case bound on the symmetry coefficient.

\paragraph{}  The paper is organized as follows. Section~\ref{sec:radiiSymmetry} starts with the definition of the different radii that appear in the course of the paper along with some basic properties. Then, Section~\ref{sec:MinkAsymmDefSub} introduces variants of symmetry measures that we use in the subsequent sections. The remainder of the paper is organized in groups along the individual theorems in the section headings that are generalized.

\section{Radii Definitions and Preliminaries} 
\label{sec:radiiSymmetry}
\paragraph{} Before giving the radii definitions, we briefly explain our notation.

\paragraph{} Throughout this paper, we are working in $d$-dimensional real space $\R^d$\index{$\R^d$} and for $A \subseteq \R^d$ we write $\lin(A)$\index{linear hull}\index{$\lin(\cdot)$}, $\aff(A)$\index{affine hull}\index{$\aff(\cdot)$}, $\conv(A)$\index{convex hull}\index{$\conv(\cdot)$}, $\int(A)$\index{interior}\index{$\int(\cdot)$}, $\relint(A)$\index{relative interior}\index{$\relint(\cdot)$}, and $\bd(A)$\index{boundary}\index{$\bd(\cdot)$} for the linear, affine, or convex hull and the interior, relative interior and the boundary of $A$, respectively. For two points $x,y \in \R^d$, we abbreviate $[x,y]:= \conv\{x,y\}$.

The \emph{dimension} of a set $A \subset \R^d$  is the dimension of the smallest affine subspace containing it.
Furthermore, for any two sets $A, B \subset \R^d$ and $\rho \in \R$, let $\rho A := \{\rho a: a \in A\}$ and $A+B:= \{a+b: a\in A, b\in B\}$ the $\rho$-\emph{dilatation} of $A$ and the \emph{Minkowski sum} of $A$ and $B$, respectively. 
We abbreviate  $A + (-B)$ by $A -B$ and  $A+\{c\}$ by $A+c$.
A set $A \subset \R^d$ 
is called \emph{0-symmetric} if $-A=A$.
If there is a $c \in \R^d$ such that $-(c+A)= c+A$, we call $A$ \emph{symmetric}. 

\paragraph{} For two vectors $x,y\in \R^d$, we use the notation $x^Ty:= \sum_{i=1}^d x_i y_i$ 
for the standard scalar product of $x$ and $y$, and by 
$H_{\leq}(a,\beta):= \{x \in \R^d: a^T x \leq \beta\}$
we denote the half-space induced by $a\in \R^d$ and $\beta \in \R$, bounded by the hyperplane $H_{=}(a,\beta):= \{x \in \R^d: a^T x = \beta\}$. 

For a vector $a \in \R^d$ and a convex set $K \subset \R^d$, we write
$ h(K,a):= \sup \{a^T x: x \in K\}$ for the \emph{support function} of $K$ in direction $a$.

\paragraph{} A non-empty set $K \subset \R^d$ which is convex and compact is called a \emph{convex body}. 
We write $\CC^d$ for the family of all convex bodies and $\CC^d_0$ for the family of all fulldimensional convex bodies in $\R^d$. Further, we write $\ext(K)$, $\rec(K)$, and $\ls(K)$ for the set of  \emph{extreme points} of $K$, the \emph{recession cone} of $K$, and the \emph{lineality space} of $K$, respectively.

If a polytope $P$ is described as a bounded intersection of halfspaces, we say that $P$ is in $\CH$-presentation. If $P$ is given as the convex hull of finitely many points, we call this a $\CV$-presentation of $P$.
In both cases, the representation is called \emph{rational}, if all vectors given in the representation are rational. 
A \emph{simplex} is the convex hull of $d+1$ affinely independent points. 

We write $\B_2:= \{x \in \R^d: \|x\|_2 \leq 1\}$ for the unit ball of the Euclidean norm $\|\cdot\|_2$ in $\R^d$ and $\S_2:=\{x \in \R^d: \|x\|_2=1\}$ for the respective unit sphere.

Finally, for any $k \in \N$, we abbreviate $[k]:=\{1,\dots,k\}$.

\subsection{Radii Definitions}
\label{sec:radii}

\paragraph{} We start this section by defining the circumradius of a closed convex set $K\subseteq \R^d$ with respect to some gauge body $C\subseteq\R^d$. The circumradius appears at many points throughout this paper and also serves for the definition of other radii and symmetry coefficients. Note that in all the following definitions $C$ is not necessarily assumed to be symmetric. 

\begin{defi}[$C$-radius] \label{defi:Cradius}%
Let  $K, C \subseteq \R^d$ non-empty, closed, and convex. We denote by $R(K,C)$ the least dilatation factor $\rho \geq 0$, such that a translate of $\rho C$ 
contains $K$, and call it the \emph{$C$-radius} of $K$ (cf. Figure~\ref{fig:RKC}). In mathematical terms,

\begin{equation}
R(K,C):= \inf\{\rho\geq 0: \exists c \in \R^d ~s.t.~ K \subseteq c+\rho C\}.
\label{eq:Cradius}
\end{equation}
\end{defi}

\begin{figure}[h]
\centering
\includegraphics[width=1\textwidth]{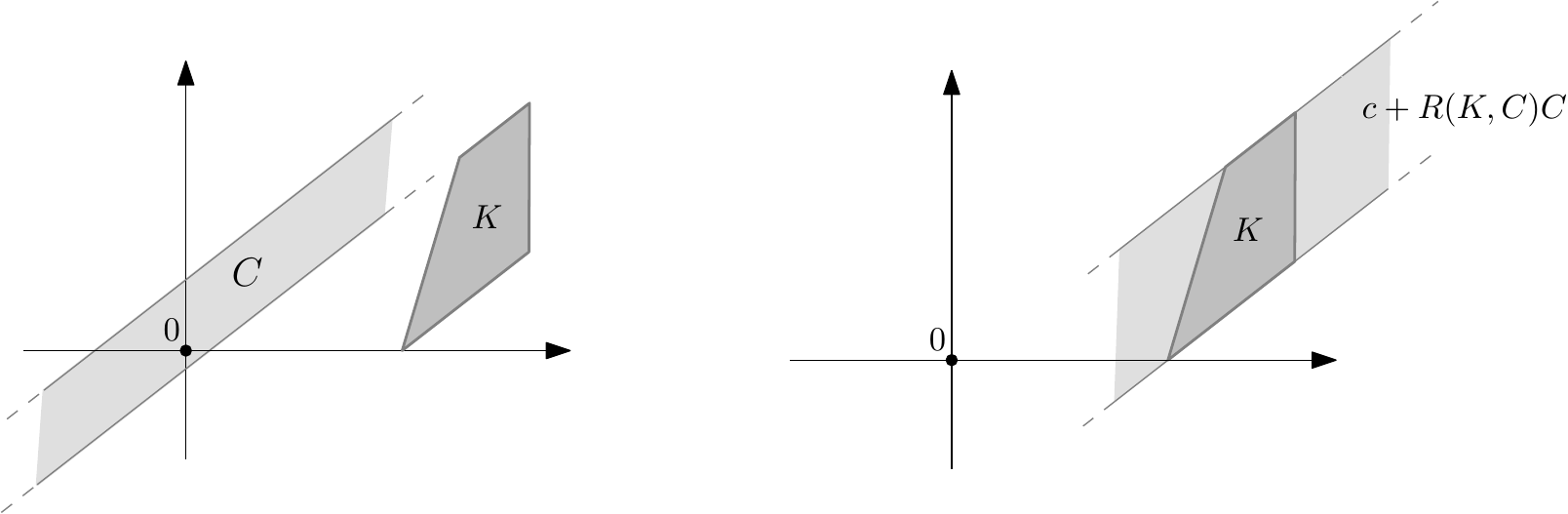}
\caption[$C$-Radius of a convex body $K$]{The $C$-radius of a convex body $K\subseteq \R^2$. Left: The convex body $K$ and an unbounded closed convex set $C$. Right: A minimally scaled copy of $C$ is translated such that it contains $K$. \label{fig:RKC}}
\end{figure}
\paragraph{} If $C = \B_2$ is the Euclidean ball, $R(K,\B_2)$ is the common Euclidean circumradius of $K$. If $C$ is 0-symmetric $R(K,C)$ measures the circumradius of $K$ with respect to the norm $\|\cdot\|_C$ induced by the gauge body $C$. Since Definition~\ref{defi:Cradius} allows unbounded convex sets $K$ and $C$, one has to be careful with the cases where the infimum in \eqref{eq:Cradius} is not attained. We treat these cases in the following lemma. 

Note that by definition $R(K,C)$ is invariant under translations of $K$ and $C$. 
Hence, we may assume $0\in \relint(K) \cap \relint(C)$ without loss of generality, wherever it simplifies the notation.

\begin{lem}
Let $K,C \subset \R^d$ convex and closed with $\ext(K)$, $\ext(C)$ bounded and 
$0 \in \relint(K) \cap \relint(C)$. Then, 
\begin{enumerate}[a)]
\item $R(K,C) < \infty$ if and only if $K \subset \lin(C)$ and $\rec(K) \subset \rec(C)$, \label{it:feasible}
\item $R(K,C)= 0$ if and only if $R(K,\rec(C)) < \infty$, and \label{it:0}
\item if $R(K,C) \not \in \{0,\infty\}$, there exists a center $c\in \R^d$ such that $K \subseteq c + R(K,C)C$.\label{it:centerExists}
\end{enumerate} 
\label{lem:radiiAttained}
\end{lem}
\begin{proof}
Let $K_1:= \conv (\ext(K))$ and $C_1:= \conv(\ext(C))$ such that $K$ and $C$ can be expressed as $K= K_1 +\rec(K)$ and $C= C_1+\rec(C)$, respectively.

\begin{enumerate}[a)]
\item If $R(K,C) < \infty$, there exist $c \in \R^d$, $\rho \geq 0$ such that $K \subseteq c +\rho C$. This implies the right hand side in~\ref{it:feasible}).
If, on the other hand, $K \subseteq \lin(C)$ and $\rec(K) \subseteq \rec(C)$, we immediately obtain $K_1\subseteq \lin(C)$ and since $K_1$ is bounded and $0 \in \relint(C)$, there exists 
$\rho > 0$ such that $K_1 \subseteq \rho C$. Moreover, since $\rec(K) \subseteq \rec(C) = \rho \ \rec(C)$ we obtain $K = K_1 +\rec(K) \subseteq  \rho C$.

\item Assume that $R(K,C)=0$. Then, by \ref{it:feasible}), $\rec(K)\subseteq \rec(C)$. If $R(K, \rec(C))=\infty$, then \ref{it:feasible}) implies the existence of a point $x \in K$ such that $x\not \in \lin(\rec(C))$. Now, assume without loss of generality that $C_1 \subset \B_2$. Thus, $c+\rho C = c + \rho C_1 + \rec(C) \subset c+ \rho \B_2 + \lin(\rec(C))$ for all $c \in \R^d$ and $\rho > 0$.  Denote the Euclidean distance of $x$ to $\lin(\rec(C))$ by $\bar \rho >0$. Since $x,0\in K$,  we conclude that 
$K \subset c + \rho C$ is possible only if $\rho \ge \frac{\bar\rho}{2} > 0$, 
which contradicts the assumption.

If, on the other hand, there exist $c \in \R^d$ and $\rho^* \ge 0$ such that $K \subseteq c+ \rho^* \rec(C)$, then $K \subset c + \rho \ \rec(C) \subset c+\rho C$ for all $\rho>0$ and therefore $R(K,C)=0$.

\item Since $R(K,C) \in (0,\infty)$, Part~\ref{it:feasible}) and \ref{it:0}) imply $\rec(K) \subseteq \rec(C)$ and $K_1\not \subseteq \lin(\rec(C))$.
Hence there exists $\rho>0$ such that $R(K,C)= R(K_1,C) =R(K_1, C\cap \rho \B_2)$
and therefore by the Blaschke selection theorem \cite[Theorem 1.8.6]{schneider} some  $c \in \R^d$ such that $K_1\subseteq c+ R(K,C)(C\cap \rho_2 \B_2)\subseteq c+R(K,C)C$. Because of $\rec(K)\subseteq \rec(C)$, this implies $K \subseteq c+R(K,C)C$.
\end{enumerate}
\end{proof}

\paragraph{} As an immediate corollary of Lemma~\ref{lem:radiiAttained}, we obtain the following if $K$ and $C$ are bounded.
\begin{cor}
Let $K,C \in \CC^d$ with $0 \in \relint(K) \cap \relint(C)$. Then, 
\begin{enumerate}[a)]
\item $R(K,C) < \infty$ if and only if $K \subset \lin(C)$,
\item $R(K,C)= 0$ if and only if $K$ is a singleton, and
\item if $R(K,C) \neq \infty$, there exists a center $c\in \R^d$ such that $K \subseteq c + R(K,C)C$.
\end{enumerate} 
\label{cor:radiiAttainedBounded}
\end{cor}

\paragraph{} In the same way as the circumradius, we introduce the inradius of a convex body $K$ 
with respect to a gauge body $C$. 

\begin{defi}[$C$-inradius]\label{defi:CinRadius}%
Let $K, C\subseteq \R^d$ non-empty, closed, and convex. Then, the \emph{$C$-inradius} $r(K,C)$ of $K$ is the greatest scaling factor $\rho \geq 0$, such that a translate of $\rho C$ is contained in $K$. 
In other words:
\begin{equation*}
r(K,C):= \sup\{\rho\geq 0: \exists c \in \R^d ~s.t.~ c + \rho C \subseteq K\}
\end{equation*}
\end{defi}

\paragraph{} Strictly speaking, there is no need to introduce $r(K,C)$ since it can easily be expressed as
\begin{equation}
r(K,C)= R(C,K)^{-1} ,
\label{eq:inCircumRadius}
\end{equation}
using the conventions $\infty^{-1} = 0$ and $0^{-1}= \infty$ (cf. e.g.~\cite[Section 4.1.2]{hadwigerBuch}). Nevertheless, we keep the notation, as the little $r$, reminiscent of \emph{inradius}, emphasizes the resemblance with the theorems being generalized in the following.

\paragraph{} Whereas the definitions of in- and circumradius are canonical even for asymmetric $C$, there exists more than one generalization of the diameter (see e.g.~\cite{dgk-63,leichtweiss}). At least for our purposes, the following definition seems the most advantageous. 

\begin{defi}[$C$-diameter]
Let $K,C\subseteq \R^d$ non-empty, closed, and convex. We define
 $$ R_1(K,C):= \sup \bigl\{R\bigl([x,y], C\bigr) : x,y\in K\bigr\}$$
as the $C$-radius of the \enquote{longest} segment in $K$ and 
$$D(K,C):= 2R_1(K,C)$$
as the \emph{$C$-diameter} of $K$ (cf. Figure~\ref{fig:diameter}).
\label{defi:Cdiameter}
\end{defi}

\paragraph{} The notation as $R_1(K,C)$ expresses the diameter as the biggest circumradius of 1-dimensional subsets of $K$ and is consistent with the more general core-radii introduced in \cite{coresetPaperDCG}.

\begin{figure}[h]
\centering
\includegraphics[width=0.9\textwidth]{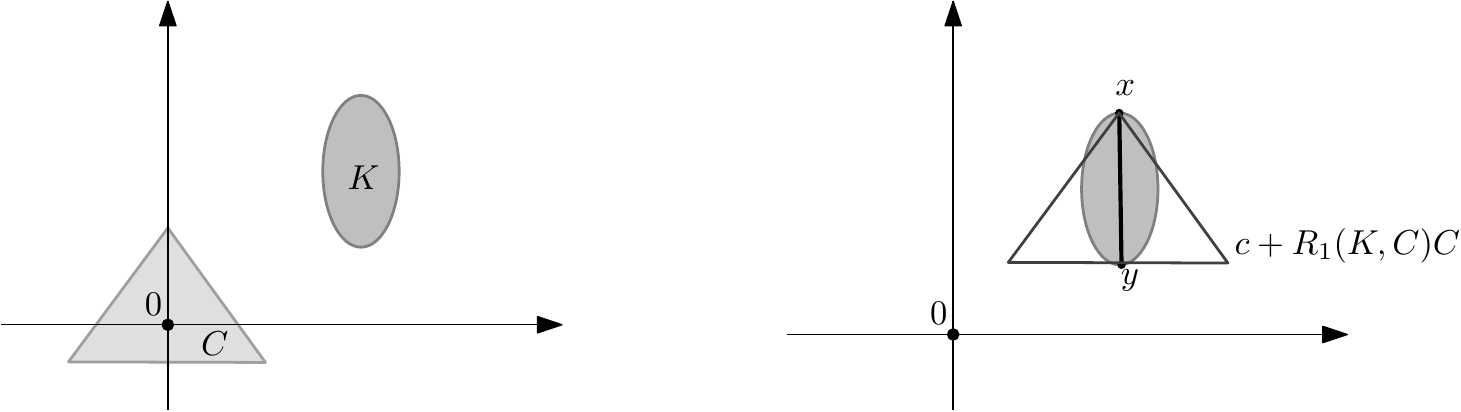}
\caption[The $C$-diameter $K$.]{The $C$-diameter of $K$. Left: $K, C \subseteq \R^2$.  Right: The indicated segment $[x,y]$ has maximal $C$-radius among all line segments contained in $K$. \label{fig:diameter}}
\end{figure}

\paragraph{} Analogously, we define the width for a closed and convex set $K \subseteq \R^d$ with respect to a general gauge body $C\subseteq \R^d$. The idea is to measure the ratio of distances of two parallel hyperplanes that sandwich $K$ and $C$, respectively (cf. Figure~\ref{fig:Cwidth}).


\begin{defi}[$C$-width] \label{defi:s-breadth-and-width}%
Let $K,C \subseteq \R^d$ non-empty, closed, and convex. 

Using the convention that $\frac \alpha \beta := \infty$, 
whenever $\alpha = \infty$ or $\beta = 0$ and $\frac \alpha \beta := 0$,  if $\alpha \neq \beta = \infty$, 
we define

\begin{equation}
r_1(K,C):= \inf \left \{\frac{h(K-K, a)}{h(C-C, a)}: a \in \R^d \setminus\{0\} \right\} .
\label{eq:width}
\end{equation}
and denote by \[w(K,C):= 2 r_1(K,C)\]
the \emph{$C$-width} of $K$.
\end{defi}
\begin{figure}[h]
\centering
\includegraphics[width=0.9\textwidth]{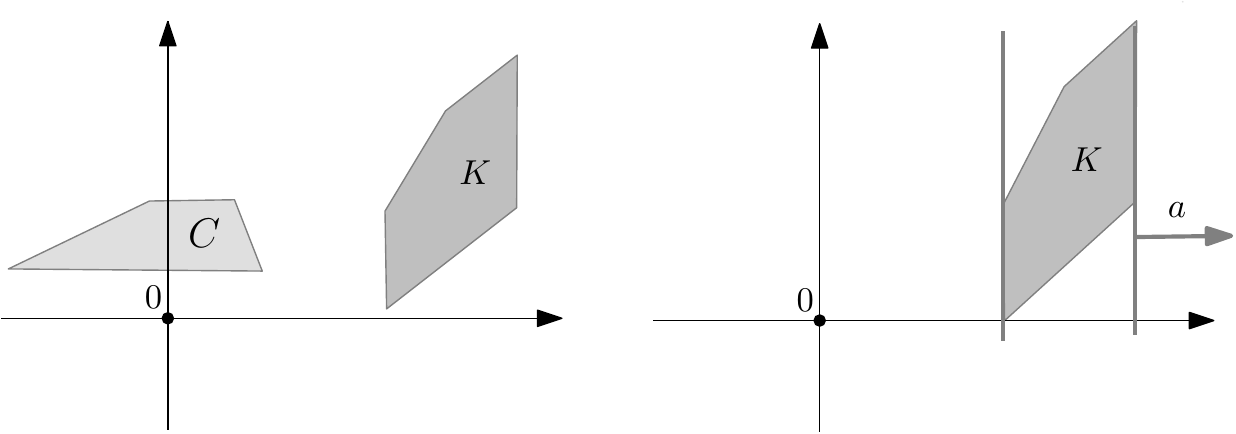}
\caption[The symmetric $C$-width of a convex body $K$]{The $C$-width of $K$. Left: $K,C\subseteq \R^2$. Right: The $C$-width of $K$ is attained for a direction $a$, for which the ratio $h(K-K,a)/h(C-C,a)$ is minimal.\label{fig:Cwidth}}
\end{figure}

\paragraph{} Again, in case $C$ is symmetric, $w(K,C)$ is the (minimal) width of $K$ with respect to $\|\cdot\|_C$ in the usual sense. 


\begin{rem}[Pathological cases] \label{rem:pathological}%
In case that $R(K,C) = \infty$ or equivalently $r(C,K)=0$ the values of $R_1(K,C)$ and 
  $r_1(C,K)$ can take any value within $[0,\infty]$ independently of the former ones
\footnote{Consider e.g.~$K=[0,\infty) \times [-\rho,\rho]$ with 
  $\rho \in [0,\infty]$ and $C=(-\infty,0] \times [-1,1]$.}. However, 
  restricting to $rec(K)=\ls(K)$ and $\rec(C)=\ls(C)$ 
  (which, i.e., is the case if $K$ and $C$ are symmetric) we have
\[  \Bigl\{R(K,C), R_1(K,C), r(C,K), r_1(C,K)\Bigr\} \cap \{0,\infty\}\neq \emptyset \]
\[ \Longrightarrow \quad R(K,C)=R_1(K,C) = r(C,K)^{-1} = r_1(C,K)^{-1} . \]
\end{rem}

\paragraph{} Our first observation is that both the $C$-width and the $C$-diameter remain unaffected if the arguments are symmetrized. This fact allows us to establish a useful identity relating  $R_1(K,C)$ to $r_1(C,K)$.

\begin{lem}[Invariance under symmetrization]\label{lem:diameterWidth}%
Let $K,C \subseteq\R^d $ non-empty, closed, and convex. The following three identities hold

\begin{enumerate}[a)]
\item $\displaystyle r_1(K,C)= r_1\left(\frac12 (K-K), \frac12(C-C)\right)$, \label{it:wSymm} 
\item $\displaystyle R_1(K,C)= R_1\left(\frac12 (K-K), \frac12(C-C)\right)$, and \label{it:Dsymm}
\item $r_1(K,C)=R_1(C,K)^{-1}$ \\
(or equivalently, for the non-pathological cases, $D(K,C)w(C,K)= 4$). \label{it:Dw}
\end{enumerate}
\end{lem}
\begin{proof}
First, observe that convex $K\subseteq \R^d$, we have 

$$K-K= \frac12 (K-K) - \frac12 (K-K)$$
Using this identity, \ref{it:wSymm}) follows immediately from the definition of the $C$-width via Equation \eqref{eq:width}.

For the proof of \ref{it:Dsymm}),
let $A \in \{K,C\}$ and $p,q \in A$. Then $p - \frac12(p+q) = \frac12(p-q) \in \frac12(A-A)$ and  $q - \frac12(p+q) = \frac12(q-p) \in \frac12(A-A)$. Thus,  with $A=K$, we obtain that
$R_1(K,C) \le R_1(\frac12(K-K),C)$ and, with $A=C$, that $R_1(K,C) \ge R_1(K,\frac12(C-C))$.

On the other hand, let $p=\frac12(x_p-y_p),q=\frac12(x_q-y_q) \in \frac12 (A-A)$ 
with $x_p,x_q, y_p,y_q \in A$. Then $p+\frac12(y_p+y_q) =\frac12(x_p+y_q)\in A$ and $q+\frac12(y_p+y_q) = \frac12(x_q+y_p) \in A$. Hence it follows $R_1(\frac12(K-K),C) \le R_1(K,C)$ from using $A=K$ and 
$R_1(K,\frac12(C-C)) \ge R_1(K,C)$ from using $A=C$.

For Part~\ref{it:Dw}), we use the well known identities $R_1(K,C)=R(K,C)$ and $r_1(K,C)=r(K,C)$ 
  for symmetric $K$ and $C$ (e.g.~\cite[(1.3)]{gk-92}) and obtain 

{\renewcommand{\arraystretch}{2}
$$\begin{array}{rl}
& R_1(K,C)\stackrel{\ref{it:Dsymm})}{=}R_1\left(\frac12(K-K),\frac12(C-C)\right)
=R\left(\frac12(K-K),\frac12(C-C)\right) \\
\stackrel{\eqref{eq:inCircumRadius}}{=}& r\left(\frac12(C-C),\frac12(K-K)\right)^{-1} 
=r_1\left(\frac12(C-C),\frac12(K-K)\right)^{-1}\stackrel{\ref{it:wSymm})}{=}r_1(C,K)^{-1}.\\
\end{array}$$}
\end{proof}

\subsection{Some specific radii}
\paragraph{} We conclude this section of preparing lemmas by computing some radii of certain convex bodies that will serve to show the tightness of several inequalities in the sequel. Figure~\ref{fig:computingRadii} illustrates the bodies apprearing in Lemmas~\ref{lem:TalphaT} to \ref{lem:convTcupB}.

\begin{lem}[Partial difference bodies of simplices]
Let $S \subseteq \R^d$ a $d$-simplex and $\alpha,\beta \in [0,1]$. Define $C:= S- \alpha S$, $K:=- S +\beta S$. Then,
\begin{equation} \label{it:cRTalphaT} 
R(K,C)= \frac{d+\beta}{1+d\alpha}  \quad \text{ and }  \quad R_1(K,C) = \frac{\beta+1}{\alpha+1}.
\end{equation}
\label{lem:TalphaT}
\end{lem}
\begin{proof}
Since $R(K,C)$ and $R_1(K,C)$ are invariant under translations of $K$ and $C$, we may assume that there exist $x_1,\dots, x_{d+1}, a_1,\dots, a_{d+1} \in \R^d$ such that 
$$S= \conv \{x_1,\dots, x_{d+1}\} = \bigcap_{i=1}^{d+1} H_\leq (a_i,1),$$ 
where the $a_i$ are numbered such that
$$ a_i^T x_j = \begin{cases} 1 & \text{ if } i \neq j \\ -d & \text{ if } i= j \end{cases} $$
for all $i,j \in [d+1]$.

In a first step we prove $-S+\beta S \subseteq \frac{d+\beta}{1+d\alpha}(S-\alpha S)$, which implies 
$R(-S +\beta S, S-\alpha S)\leq \frac{d+\beta}{1+d\alpha}$. 
For this purpose let $i,j\in [d+1]$ with $i\neq j$ such that  $-x_i + \beta x_j$  is a vertex of $-S +\beta S$. Showing that there exists $p \in S$ such that 
\begin{equation}
-x_i+ \beta x_j = \frac{d+\beta}{1+d\alpha} p - \frac{\alpha(d+\beta)}{1+d\alpha} x_i,
\label{eq:TalphaT}
\end{equation} 
implies that $-x_i +\beta x_j \in \frac{d+\beta}{1+d\alpha}(S-\alpha S)$. Rearranging \eqref{eq:TalphaT} yields that we need
$$ p= \frac{1+d\alpha}{d+\beta}\left(-x_i+ \beta x_j \right)+\alpha x_i .$$
However, with this expression, it is straightforward to verify that $a_i^T p = 1$ and $a_k^T p < 1$ for all $k\in [d+1]\setminus\{i\}$ 
and therefore that $p \in S$.

On the other hand, we have $R(-S+\beta S, S) = d+\beta$ and $h(S -\alpha S,a_i) = 1+d\alpha$ for all $i \in [d+1]$, which implies $R(-S +\beta S, S-\alpha S)\geq \frac{d+\beta}{1+d\alpha}$.

Now consider the diameter:
Since $K-K = (1+\beta)(S-S)$ and $C-C= (1+\alpha)(S-S)$, Lemma~\ref{lem:diameterWidth}\ref{it:Dsymm}) yields 
$$ R_1(K,C)= R_1\left((1+\beta)(S-S), (1+\alpha)(S-S)\right)= \frac{1+\beta}{1+\alpha}.$$
\end{proof}

\begin{figure}[h]
\centering
\includegraphics[width=0.95\textwidth]{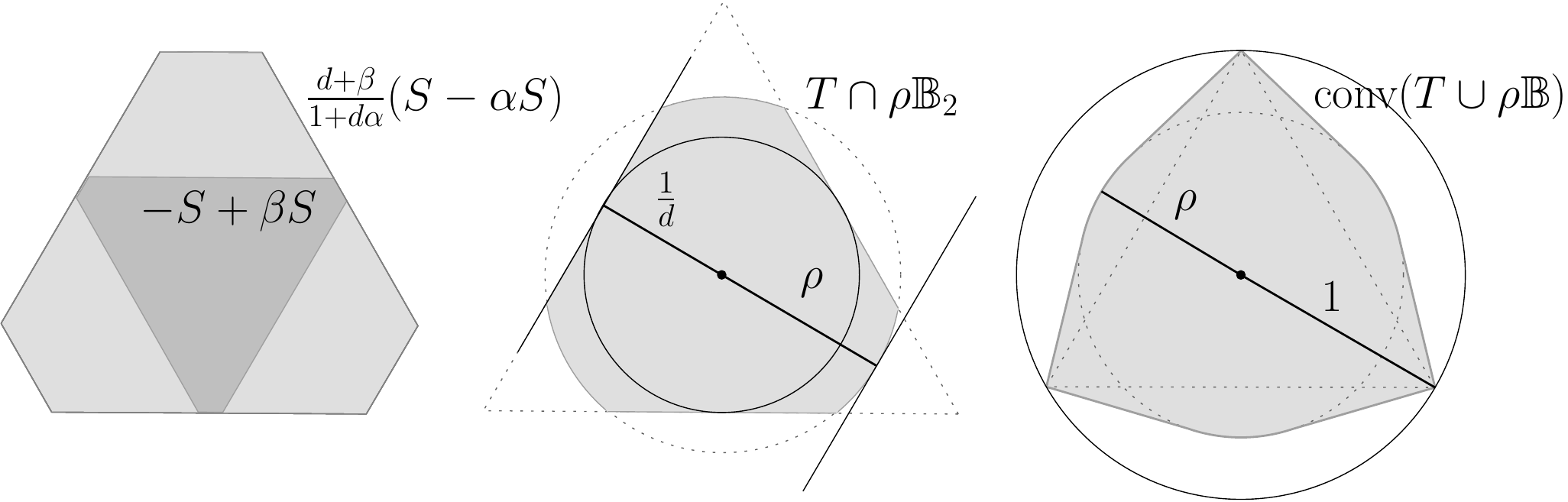}
\caption{The bodies from Lemmas~\ref{lem:TalphaT} to \ref{lem:convTcupB}. Left: Two partial differences of a (regular) simplex with $0< \alpha < \beta < 1$ and $S- \alpha S$ optimally scaled to contain $-S+\beta S$. Middle: Intersection of a regular simplex $T$ and a ball of radius $\rho$ with indications for its inradius and width. Right: Convex hull of a regular simplex $T$ and a ball of radius $\rho$ with indications for its circumradius and diameter. \label{fig:computingRadii}}
\end{figure}

\begin{lem}[Regular simplex intersected with a ball] \label{lem:simplexCapBall}%
Let $T\subseteq \B_2$ a regular simplex with all its vertices on the Euclidean unit sphere, $\rho \in\left[\frac{1}{d}, 1\right]$, and $K:= T\cap \rho \B_2.$
Then,
$$ R(-K, K)= d\rho, \quad r(K,\B_2)= \frac{1}{d}, \quad \text{and} \quad  r_1(K,\B_2)= \min \left\{r_1(T,\B_2), 
\frac{1}{2}\left(\rho +\frac{1}{d}\right)\right\}.$$
If further $C=  T\cap \rho_2 \B_2$ with $\rho_2\leq \rho$. Then,
$$ R(K,C)= \frac{\rho}{\rho_2} \quad \text{ and } \quad R(C,K)=1.$$
\end{lem}
\begin{proof}
 As $\rho\geq \frac{1}{d}$ and $r(T, \B_2)=\frac{1}{d}$,
 $$-K \subseteq \rho \B_2  \subseteq d\rho T \cap d \rho^2 \B_2= d\rho K.$$
Again, since $r(T, \B_2)=\frac{1}{d}$, this scaling is best possible. 
Hence, $R(-K, K)= d\rho$.

Further, since $\rho\geq \frac{1}{d}$, $r(K, \B_2)= r(T,\B_2)= \frac{1}{d}$. And, if $r_1(K,\B_2)< r_1(T, \B_2)$, then, because of $\rho\geq \frac{1}{d}$, the width of $K$ is attained between a pair of hyperplanes supporting $T$ in a point $x$ in the relative interior of a facet of $T$ and $-\rho x$, respectively. Hence,
$$r_1(K,\B_2)= \min \left\{r_1(T,\B_2),\frac{1}{2} \left(\rho +\frac{1}{d}\right)\right\}.$$ 

For the  second statement, we immediately obtain $R(K,C)= R(\rho\B_2,\rho_2\B_2)= \frac{\rho_2}{\rho}$ by the definiton of $K$ and $C$.
And finally, since $\rho_2 \leq \rho$, $C \subseteq K$ and $C$ touches all facets of $T$. Since these are also facets of $K$, $R(C,K)=1$ by Corollary 2.4 and Theorem 2.3 in \cite{coresetPaperDCG}. 
\end{proof}

\begin{lem}[Convex hull of a regular simplex and a ball]\label{lem:convTcupB}%
Let $T\subseteq \B_2$ a regular simplex with all its vertices on the Euclidean unit sphere, $\rho \in \left[\frac{1}{d}, 1\right]$ and $K:= \conv\bigl(T\cup \rho \B_2\bigr).$ Then,
$$ R(-K, K) =\frac{1}{\rho},\quad R(K,\B_2)=1 \quad \text{ and } \quad R_1(K,\B_2)= \max\left\{R_1(T, \B_2), \frac{1+\rho}{2}\right\}.$$
\end{lem}

\begin{proof}
We have $\frac{1}{\rho} K= \conv\left(\frac{1}{\rho}T \cup \B_2\right)$. Since $-T \subseteq \B_2$ and $\rho\leq 1$, it follows that $-K \subseteq \frac{1}{\rho} K$. Optimality of this inclusion is easily verifiable by \cite[Theorem 2.3]{coresetPaperDCG}, since $\mathrm{ext}(T) \subseteq \S_2$. This shows $R(-K, K) =\frac{1}{\rho}$.
Further, by definition of $K$ we have $R(K,\B_2)=1$. 

If $R_1(K, \B_2) > R_1(T, \B_2)$, then, because of $\rho \leq 1$, the diameter of $K$ is attained between a vertex $x$ of $T$ and $-\rho x$. Hence,
$$R_1(K,\B_2)= \max\left\{R_1(T, \B_2), \frac{1+\rho}{2}\right\}.$$
\end{proof}

\section{Asymmetry Measures}
\label{sec:MinkAsymmDefSub}

\subsection{Minkowski Asymmetry}
\paragraph{} There is a rich variety of measurements for the asymmetry of a convex body; see \cite{gruenbaum-63} (and in particular Section 6) 
for an overview. It is already claimed in \cite{gruenbaum-63}  that the one which has received most interest is Minkowski's measure of symmetry.
Its reciprocal measures the extent to which $K$ needs to be scaled in order to contain a translate of $-K$ 
(cf.~\cite[Notes for Section 3.1]{schneider}), which in our terminology, is the $K$-radius of $-K$. 
For short, We call the latter value, being large for \enquote{very asymmetric} sets, the \emph{Minkowski asymmetry} of $K$.

\begin{defi}[Minkowski asymmetry]\label{defi:MinkowskiAsymm}%
Let $K\subseteq \R^d$, non-empty, closed, and convex. We denote by
\begin{equation} \label{eq:minkowskiAsymm}
 s(K):= R(-K,K)
\end{equation}
the \emph{Minkowski asymmetry} of $K$.

Further, if $c \in \R^d$ is such that $-(K-c) \subseteq s(K)(K-c)$, we call $c$ a \emph{Minkowski center} of $K$, and if $0$ is a Minkowski center of $K$, we say that the body $K$ is \emph{Minkowski centered} (cf. Figure~\ref{fig:minkAsymm}).
\end{defi}

\begin{figure}[h]
\centering
\includegraphics[width=0.9\textwidth]{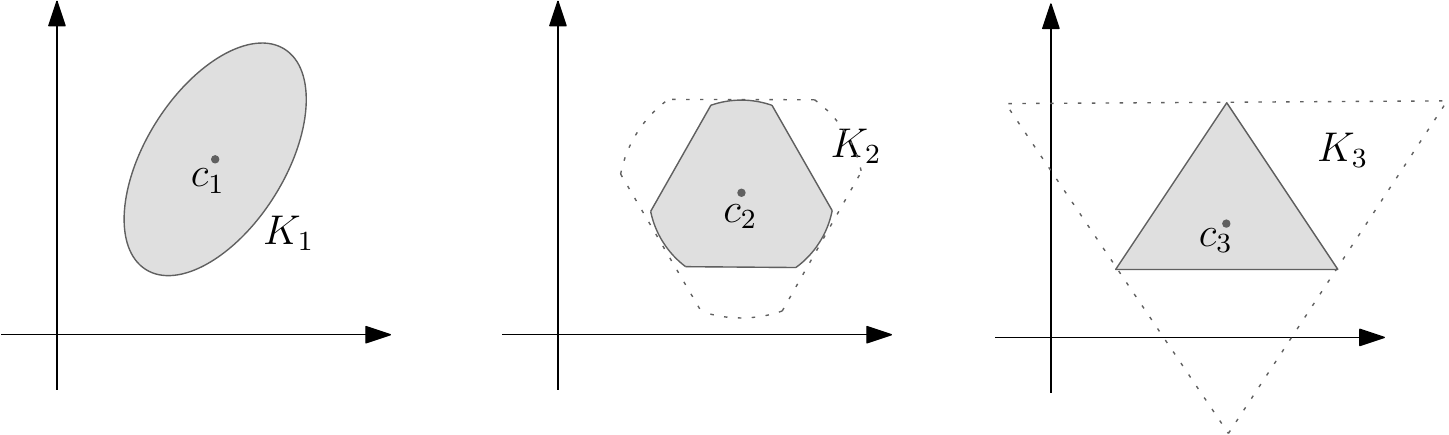}
\caption[Planar examples with different Minkowski asymmetry]{Planar examples with different Minkowski asymmetry. Left: $K_1$ is symmetric, $s(K_1)=1$ and its Minkowski center is $c_1$. Middle: $K_2$ with $s(K_2)=3/2$ and Minkowski center $c_2$. Right: A 2-simplex $K_3$ with 
  $s(K_3)=2$ and Minkowski center $c_3$. The suitable homothetics of $-K_2$ and $-K_3$ containing $K_2$ and $K_3$, respectively, are indicated in dotted gray. \label{fig:minkAsymm} }
\end{figure}

\paragraph{} In all three examples in Figure~\ref{fig:minkAsymm}, the Minkowski center of $K_i$ is contained in $K_i$, $i=1,2,3$, a property which is also true in general as the following lemma shows.

\begin{lem}[Minkowski center is inside $K$]
Let $K \subseteq \R^d$ non-empty, closed, and convex and  $c \in \R^d$ a Minkowski center of $K$. Then,
$$c \in \relint(K).$$
\end{lem}
\begin{proof}
Without loss of generality we may assume $\int(K) \neq \emptyset$ and $c=0$. For a contradiction suppose $0 \notin \int(K)$. 
  Then there exists $a \in \R^d\setminus\{0\}$ such that $a^T x \leq 0$ for all $x\in K$. Since $-K \subseteq s(K)K$, we obtain $a^T x = 0$ for all $x\in K$, which contradicts $\int(K) \neq \emptyset$.
\end{proof}

\paragraph{} For unbounded $K$, the following statement can be obtained from Lemma~\ref{lem:radiiAttained} (cf.~\cite[Appendix A]{bf-08})

\begin{rem}[Asymmetry for unbounded convex sets]
We have $R(-K,K)= 0$ if and only if $K$ is an affine subspace and $R(-K,K)= \infty$ if and only if $\rec(K)$ is not
a linear subspace. The latter means that cylinders $K= K_1 + F$, with $F$ a linear subspace and 
$K_1 \subset F^\bot$ a non-singleton compact convex set, are the only unbounded sets with 
Minkowski asymmetry different from $0$ and $\infty$ and for them $s(K)= s(K_1)$ holds.
\label{rem:infiniteSymmetry}
\end{rem}

\paragraph{} Because of Remark~\ref{rem:infiniteSymmetry} we henceforth assume $K \in \CC^d$.

\paragraph{} In contrast to the three examples in Figure~\ref{fig:minkAsymm}, 
for an arbitrary $K \in \CC^d$, it can happen that the Minkowski center is not unique and even that the 
set of centers is of dimension up to $d-2$
as indicated by Figure~\ref{fig:MinkCenterNotUnique} and proved in \cite{gruenbaum-63}.

\begin{figure}[h]
\centering
\includegraphics[width=0.4\textwidth]{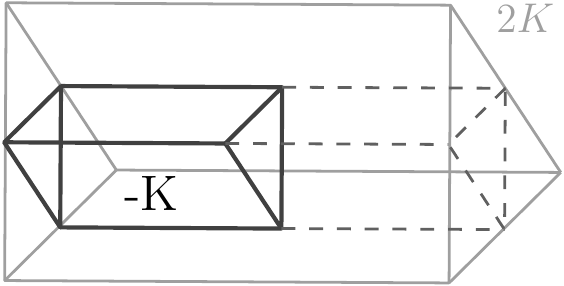}
\caption{Let $T \subseteq \R^2$ denote a regular triangle. Then, for  $K= T \times [-1,1] \subseteq \R^3$,} 
  we enforce $s(K)\geq 2$ by Proposition~\ref{prop:asymm}. However, in direction of the third coordinate the dilatation is twice as much as needed and therefore the Minkowski center of $K$ is not unique.\label{fig:MinkCenterNotUnique}
\end{figure}

\paragraph{}  The following proposition states the well-known bounds on $s(K)$. A proof in the notation that is used here can be found in \cite{coresetPaperDCG}. 

\begin{prop}[Bounds on the Minkowski asymmetry] \label{prop:asymm}%
For $K \in \CC^d$, 
$$ 1 \leq s(K) \leq d,$$ 
with $s(K)=1$ if and only if $K$ is symmetric, and $s(K)=d$ if and only if $K$ is a $d$-simplex.

\end{prop}

\paragraph{} Next, we turn to the computability of the Minkowski asymmetry. 

For an introduction to the study of the computational complexity of radii and containment problems, 
we refer to \cite{brandenbergRoth-10}, \cite{eaves-freund-82}, \cite{freund-orlin-85}, \cite{gk-93}.
The following Lemma may be derived from the above references or the explicit proof in \cite{bf-08}:

\begin{lem}[Computability]
Let $K \in \CC^d$ be a rational polytope given in $\CH$- or $\CV$-presentation.
Then $s(K)$ and a Minkowski center $c \in \R^d$ such that $-(K-c) \subseteq s(K)(K-c)$ can be computed in polynomial time. 
\label{lem:computeSC}
\end{lem}
\begin{proof}
Using \eqref{eq:Cradius}, the computation of $s(K)= R(-K,K)$ requires the solution of the following optimization problem: 
\begin{equation} \label{eq:CradiusLP}
\begin{array}{rl}
 \min &  \rho \\
 s.t.	&  -K \subseteq c + \rho K \\
			&	 c \in \R^d \\
			&	 \rho \geq 0.
\end{array}
\end{equation}

By Proposition~\ref{prop:asymm}, and Lemma~\ref{lem:radiiAttained}\ref{it:centerExists}), there exists a solution $(c^*,\rho^*) \in \R^d \times [1,d]$

of \eqref{eq:CradiusLP}.  By definition, $s(K)= \rho^*$ and we have that $c= -\frac{1}{s(K)+1} c^*$ is a Minkowski center of $K$, as
$$ - \left(K+\frac{1}{s(K)+1} c^*\right) \subseteq - \frac{1}{s(K)+1}c^* + (c^*+ s(K) K) =s(K)\left( K+ \frac{1}{s(K)+1}c^* \right).$$
Now, \cite{brandenbergRoth-10} demonstrates that the computation of $R(K,C)$ amounts to solving a Linear Program if $K$ and $C$ are both given 
in $\CH$-presentation or both given in $\CV$-presentation. 
Hence, in both cases, $s(K)=R(-K,K)$ and a respective Minkowski center can be computed in polynomial time.
\end{proof}

\paragraph{} Lemma~\ref{lem:computeSC} can also be used to decide whether a polytope $K$ in $\CH$- or $\CV$-presentation is symmetric and to compute its center of symmetry in this case. This yields an alternative proof for \cite[Theorem 2.2]{gk-93}.

\subsection{John and Loewner Asymmetry}
\paragraph{} We also consider centered versions of asymmetry of a convex body $K$, i.e.\ we are interested in the minimal dilatation factor needed to cover $-(K-c_0)$ with a copy of $K-c_0$ for some $c_0 \in \R^d$ depending on $K$, but not free to be chosen for the optimal covering. 
For a general study of symmetry values as a function of $c_0$, we refer to \cite{bf-08}. Here, we focus on the presumably most natural  choices, the center of the maximum volume ellipsoid inscribed $K$ and the center of the minimal volume ellipsoid containing $K$.
Measuring the symmetry of $K$ around these centers nicely interacts with John's Theorem \cite{john}: on the one hand, the classic formulation of John's Theorem can be used to bound this centered asymmetries of a body as in Corollary~\ref{cor:boundsOnS0}. On the other hand, we will use the centered asymmetries in Theorem~\ref{theo:strongJohn} to sharpen John's Theorem itself. Because of its importance in this context, we give an explicit statement of John's Theorem in Proposition~\ref{prop:john} and refer to \cite{keithBall-john, ball-intro2CG, gruber-schuster} for proofs.

When talking about John's Theorem, we usually assume that $K$ is full dimensional, i.e.\ without loss of generality $K \in \CC_0^d$.  One may use the usual identification $\aff(K)\cong\R^{\dim(K)}$ to extend the results to lower-dimensional bodies.

\begin{prop}[John's Theorem]
For any $K \in \CC_0^d$ there exists a unique ellipsoid of maximal volume contained in $K$, which is $\B_2$
if and only if 
\begin{enumerate}[(1)]
\item $\B_2 \subseteq K$, and 
\item for some $k \in \left\{d+1,\dots, \frac{d(d+3)}{2}\right\}$, there are points $p_1,\dots, p_k \in \bd(K) \cap \S_2$ and scalars $\lambda_1,\dots, \lambda_k > 0$ such that
\begin{equation}
0= \sum_{i=1}^k \lambda_i p_i \quad \text{ and } \quad I= \sum_{i=1}^k \lambda_i p_ip_i^T \ .
\label{eq:john}
\end{equation}
\end{enumerate} 
Moreover, if $\B_2$ it the ellipsoid of maximal volume contained in $K$, then $K \subseteq d \, \B_2$ in general and $K \subseteq \sqrt{d} \, \B_2$, if $K$ is 0-symmetric.
\label{prop:john}
\end{prop}

\begin{defi}[John asymmetry] 
Let $K \in \CC_0^d$ and $c_K$ the center of the ellipsoid of maximal volume contained in $K$.
We define
$$ s_0(K):= \min \{\rho \geq 0: -(K -c_K) \subseteq \rho(K - c_K)\}$$
as the asymmetry of $K$ around the center of its maximum volume inscribed ellipsoid and call it \emph{John asymmetry}. 

\label{defi:centeredAsymmetry}
\end{defi}

\paragraph{} As already mentioned, one may use John's Theorem to obtain the same bounds on $s_0(K)$ as on $s(K)$ (cf.~\cite[p. 248]{gruenbaum-63}).

\begin{cor}[Bounds on the John asymmetry]
Let $K \in \CC_0^d$. Then,
$$ 1 \leq s_0(K) \leq d$$
with equality if and only if $K$ is symmetric in the first case and if and only if $K$ is a $d$-simplex in the latter case.

\label{cor:boundsOnS0}
\end{cor}

\paragraph{} As for the Minkowski asymmetry, the John asymmetry is computable for suitably presented polytopes.

\begin{lem}[Computability of the John asymmetry]
If $K \subseteq \R^d$ is a polytope in $\CH$-presentation, $s_0(K)$ can be approximated to any accuracy in polynomial time.

\label{lem:computeS0}
\end{lem}

\begin{proof}
First, we mention that $\aff(K)$ is efficiently computable for both representations of $K$. Hence we may assume without loss of generality that $K$ is fulldimensional.
In \cite{khachiyanTodd-complexityEllipsoid}, it is shown that for a polytope $K\subseteq \R^d$ in $\CH$-presentation, the center of the ellipsoid of maximal volume contained in  $K$ can be approximated  to any accuracy in polynomial time. An approximation of this center at hand, call it $c_K$, we can compute $\min\{\rho \geq 0: -(K -c_K) \subseteq \rho(K - c_K)\}$  via Linear Programming analogously to the Linear Program in the proof of Lemma~\ref{lem:computeSC}.
\end{proof}

\begin{rem}[Loewner asymmetry]
One could also measure the asymmetry of a body $K$ around its Loewner center, i.e.\ the center of the volume minimal enclosing ellipsoid of $K$. With the same arguments as for the John center, the values of this asymmetry measure are also contained in the interval $[1,d]$. Moreover, for a $\CV$-presented polytope $K \subseteq \R^d$, this center can be approximated to any accuracy in polynomial time \cite{khachiyanTodd-complexityEllipsoid} and therefore the asymmetry around the Loewner center can be approximated efficiently for $\CV$-polytopes by the same argument as in the proof of Lemma~\ref{lem:computeS0}.

\label{rem:Loewner}
\end{rem}

\section{The Inequalities of Bohnenblust and Leichtweiß}
\label{sec:bohnenblust}
\paragraph{} The present section gives generalizations of the inequalities of Bohnenblust \cite{bohnenblust-38} and Leicht\-weiß \cite{leichtweiss}
and shows that these generalizations are actually one and the same inequality unifying the two old theorems.

First, we prove a version of Bohnenblust's Inequality for general convex bodies with the ratio of the $C$-radius and $C$-diameter bounded in terms of the Minkowski asymmetry of $K$ and $C$. 

 
\paragraph{A note on pathological cases.} For all the geometric inequalities that follow, we assume $K,C \in \CC^d$. As a consequence of Proposition~\ref{prop:asymm}, all the right hand sides in the inequalities are therefore well defined. 
In view of Remark~\ref{rem:pathological}, one may at least 
extend the validity of these inequalities to the cases with $\ext(K),\ext(C)$ bounded and 
$rec(K)=\ls(K)$ and $\rec(C)=\ls(C)$ by presuming the ratios $0/0$ or $\infty/\infty$ to be 1 here.

\begin{theo}[Sharpening Bohnenblust's Inequality] \label{theo:bohnenblust}%
Let $K,C\in \CC^d$. Then,
\begin{equation}
 \frac{R(K, C)}{R_1(K, C)}  \leq \frac{(s(C)+1)s(K)}{s(K)+1}
 \label{eq:bohnenblust}
\end{equation}
 and for every $\sigma_K, \sigma_C \in [1,d]$, there exist bodies $K$ and $C$ with $s(K)= \sigma_K$, $ s(C)= \sigma_C$ 
such that \eqref{eq:bohnenblust} is tight for $K$ and $C$.
\end{theo}
\begin{proof}
  Suppose without loss of generality that $R_1(K,C)=1$, i.e.\ Corollary~\ref{cor:radiiAttainedBounded} ensures that for all $p_1, p_2 \in K$, 
  there is a $c\in \R^d$, such that $p_1,p_2 \in c + C$; explicitly, $p_1= c+v$ and $p_2= c+w$ with $v,w \in C$. 
  Hence $p_1-p_2= v-w \in C-C$ for all $p_1,p_2 \in K$ 
  and thus $K-K \subset C-C$. Using Proposition~\ref{prop:asymm}, 
  there exist $c_K,c_C \in \R^d$, such that
  \[c_K + \left(1+ \frac{1}{s(K)}\right)K =  K + c_K + \frac{1}{s(K)} K \subset K-K \subset C-C \subseteq 
  C + c_C + s(C)C = c_C + (1+s(C))C \] and therefore
  \[\frac{R(K,C)}{R_1(K,C)} \le \frac{s(C)+1}{1+ 1/s(K)} = \frac{(s(C)+1)s(K)}{s(K)+1} .\]

For the tightpness of the inequality, let $S \subseteq \R^d$ be a simplex, 
$\alpha := \frac{\sigma_C -d}{1-\sigma_C d}$, $\beta:=\frac{\sigma_K -d}{1-\sigma_K d}$, 
$C:= S - \alpha S$, and $K:= -S +\beta S$.
By Lemma~\ref{lem:TalphaT}
, $R(K,C)= \frac{d+\beta}{1+d\alpha}$, $s(C)= \sigma_C$ and $s(K)= \sigma_K$.
By Lemma~\ref{lem:TalphaT}
, $R_1(K,C) =\frac{1+\beta}{1+\alpha}$. Together, we obtain
$$ \frac{R(K,C)}{R_1(K,C)} = \frac{(d+\beta)(\alpha+1)}{(1+d\alpha)(\beta+1)}= \frac{(s(C)+1)s(K)}{s(K)+1}.$$
\end{proof}

\paragraph{} Because of Lemma \ref{lem:diameterWidth}, we have $R_1(K,C) = R_1(K-K,C-C) = 
R(K-K,C-C)$. Using this fact, one may also read the inequality in 
Theorem \ref{theo:bohnenblust} as an inequality between the $C$-radius of $K$ and its symmetrization
in both arguments. In this light, it is not surprising that the inequality can be tightened by bounding 
the asymmetry of the two sets.

\begin{rem}[Bohnenblust's Inequality with John asymmetry]
Since $s(K) \leq s_0(K)$ and $s(C)\leq s_0(C)$, a version of Theorem~\ref{theo:bohnenblust} 
with $s(K), s(C)$ replaced by $s_0(K), s_0(C)$ would be weaker but still valid and still 
sharpening Bohnenblust's original inequality.
As one may deduce from Proposition~\ref{prop:john}, it stays tight for the families of $K$ and $C$ as given in the proof above.
\end{rem}

\paragraph{} Note that the statement of Theorem~\ref{theo:bohnenblust} is different from the version proved by Leichtweiß in \cite{leichtweiss}. In his proof of Bohnenblust's Inequality, Leichtweiß shows an inequality which involves a different diameter definition which is strongly dependent on the position of the gauge body $C$ (cf.~\cite[Section 6]{dgk-63} for a discussion of Bohnenblust's Inequality for both diameter alternatives).

\paragraph{} Besides the fact that it is invariant under translations of $C$, the diameter/width definition which we employ has the advantage that Leichtweiß's Inequality no longer needs a seperate proof, but is the direct dual to Bohnenblust's Inequality.

\begin{cor}[Sharpening Leichtweiß's Inequality] \label{cor:leichtweiss}%
  For $K,C\in \CC^d$, we have
\begin{equation}
 \frac{r_1(K,C)}{r(K,C)} \le \frac{(s(K)+1)s(C)}{s(C)+1} 
 \label{eq:leichtweiss}
\end{equation}  
and for every $\sigma_K, \sigma_C \in [1,d]$, there exist bodies $K$ and $C$ with $s(K)= \sigma_K$, $ s(C)= \sigma_C$ such that \eqref{eq:leichtweiss} is tight for $K$ and $C$.
\end{cor}
\begin{proof}
The claim follows readily from Theorem~\ref{theo:bohnenblust} using  $r(K,C)= R(C,K)^{-1}$ (Equation \eqref{eq:inCircumRadius}) and $r_1(K,C)= R_1(C,K)^{-1}$ (Lemma~\ref{lem:diameterWidth}). For the statement about the tightness of \eqref{eq:leichtweiss}, we switch the roles of $K$ and $C$ used in the proof of the tightness of \eqref{eq:bohnenblust}.
\end{proof}

\section{The Inequalities of Jung and Steinhagen}
\label{sec:JungSteinhagen}
\paragraph{}  In the important special case that $C= \B_2$, stronger formulations of the original 
inequalities of Bohnenblust and Leichtweiß are known in the form of Jung's \cite{jung-01} and Steinhagen's \cite{steinhagen} Inequalities. However, for a body $K \in \CC^d$ with $s(K)<d$, the bounds of Theorems~\ref{theo:bohnenblust} and Corollary~\ref{cor:leichtweiss} become smaller for low values of $s(K)$ and can therefore be used to improve Jung's and Steinhagen's Inequalities. The two following theorems show that, building on symmetry coefficients, this is already the best one can obtain.

\begin{theo}[Sharpening Jung's Inequality]
Let $K \in \CC^d$. Then
\begin{equation}
  \frac{R(K, \B_2)}{R_1(K, \B_2)} \leq  \min \left\{ \sqrt{\frac{2d}{d+1}}, \frac{2s(K)}{s(K)+1} \right \}.
\label{eq:JungGeneralized}  
\end{equation}
This bound is best possible in the sense that for every value of $\sigma \in [1,d]$, there is a $K \in \CC^d$ such that $s(K)= \sigma$ and \eqref{eq:JungGeneralized} is tight for $K$.
\label{theo:jung}
\end{theo}

\begin{proof}
The inequality in \eqref{eq:JungGeneralized} follows directly from Jung's original inequality in conjunction with Theorem~\ref{theo:bohnenblust}. 
In order to show that the bound is best possible, let $\sigma \in [1,d]$, 
$T\subseteq \B_2$ a regular simplex with all its vertices on the Euclidean unit sphere, and $$K:= \conv\left(T\cup \frac{1}{\sigma} \B_2\right).$$
Then, by Lemma~\ref{lem:convTcupB}, $s(K)= \sigma $, $R(K, \B_2)=1$, and 
$$R_1(K,\B_2)= \max\left\{R_1(T, \B_2), \frac{\sigma+1}{2\sigma}\right\}.$$
Since $R_1(T,\B_2)= \sqrt{\frac{d+1}{2d}}$ by Jung's Theorem, $K$ fulfills Inequality \eqref{eq:JungGeneralized} with equality. 
\end{proof}

\begin{theo}[Sharpening Steinhagen's Inequality]
Let $K \in \CC^d$. Then
\begin{equation}
  \frac{r_1(K, \B_2)}{r(K, \B_2)} \leq  
\begin{cases}
\min \left\{~\sqrt{d},~~\frac{s(K)+1}{2}\right\} & \text{ if $d$ is odd} \\
\min \left\{\frac{d+1}{\sqrt{d+2}}, \frac{s(K)+1}{2}\right\} & \text{if $d$ is even.} \\
\end{cases}  
\label{eq:SteinhagenGeneralized}  
\end{equation}
This bound is best possible in the sense that for every value of $\sigma \in [1,d]$, there is a $K \in \CC^d$ such that $s(K)= \sigma$ and \eqref{eq:SteinhagenGeneralized} is tight for $K$.
\end{theo}
\begin{proof}
The inequality in \eqref{eq:SteinhagenGeneralized} follows directly from Steinhagens's original theorem in conjunction with Corollary~\ref{cor:leichtweiss}. In order to show that the bound is best possible, let $\sigma \in [1,d]$ and  $$K:= T\cap \frac{\sigma}{d} \B_2.$$
Then $\frac{\sigma}{d} \in \left[\frac{1}{d}, 1\right]$ and, by Lemma~\ref{lem:simplexCapBall}, 
$$s(K)= \sigma,\quad r(K,\B_2)= \frac{1}{d}\quad \text{and} \quad r_1(K,\B_2)= \min \left\{r_1(T,\B_2), \frac{\sigma+1}{2d}\right\}.$$
Thus, $K$ fulfills Inequality \eqref{eq:SteinhagenGeneralized} with equality.

\end{proof}

\section{An Inequality between In- and Circumradius}
\label{sec:FejesToth}
\paragraph{} In this section we present a generalization of a classical inequality, stating that the Euclidean circumradius of a simplex is at least $d$ times larger than its inradius. We refer to \cite[p. 28]{fejesToth1953lagerungen} for historical comments on the original authorship of the inequality itself and different proofs. Theorem~\ref{theo:fejesToth} generalizes this inequality by lower bounding the ratio of $R(K,C)$ and $r(K,C)$ in terms of $s(K)$ and $s(C)$ for arbitrary $K,C \in \CC^d$. The original inequality can be recovered from Theorem~\ref{theo:fejesToth} by choosing $C= \B_2$ and restricting $K$ to simplices.

\begin{theo}[Ratio of in- and circumradius]
Let $K,C \in \CC^d$. Then,
\begin{equation}
 \frac{R(K,C)}{r(K,C)} \geq \max\left \{\frac{s(K)}{s(C)} ,\frac{s(C)}{s(K)} \right \}.
 \label{eq:fejesToth}
\end{equation}
This bound is best-possible in the sense that for every $\sigma_K, \sigma_C \in [1,d]$, there exist $K$, $C$ such that $s(K)= \sigma_K$, $s(C)= \sigma_C$, and $K$ and $C$ fulfill \eqref{eq:fejesToth} with equality.
\label{theo:fejesToth}
\end{theo}
\begin{proof}
Since, by \eqref{eq:inCircumRadius},
 $$\frac{R(K,C)}{r(K,C)} = R(K,C)R(C,K) = \frac{R(C,K)}{r(C,K)},$$ 
it suffices to show $R(K,C)R(C,K) \ge \frac{s(K)}{s(C)}$ and we may assume without loss of generality that $C$ is Minkowski centered.
 
Because of Lemma~\ref{lem:radiiAttained}, there exist $c_i \in \R^d$, $i=1,2$, such that  $c_1 + K \subseteq R(K,C)C$ 
  and $-C \subseteq c_2 + R(C,K)(-K)$. Hence 
  \[c_1 + K \subset R(K,C)s(C)(-C) \subset R(K,C)s(C) c_2 + R(K,C)s(C)R(C,K)(-K)\]
  and thus $R(K,C)s(C)R(C,K) \geq s(K)$ by definition of $s(K)$. 
  
For the tightness of \eqref{eq:fejesToth}, let $\sigma_K, \sigma_C \in [1,d]$, $T\subseteq \B_2$ a regular simplex with all its vertices on the Euclidean unit sphere, and 
$$K:= T\cap \frac{\sigma_K}{d} \B_2 \quad \text{ and } \quad C:= T \cap \frac{\sigma_C}{d}\B_2.$$
By Lemma~\ref{lem:simplexCapBall}, $s(K)= \sigma_K$ and $s(C)= \sigma_C$. Since the roles of $K$ and $C$ are interchangeable, we can assume without loss of generality that $\sigma_K \geq \sigma_C$. Then, by Lemma~\ref{lem:simplexCapBall}, $R(K,C) = \frac{\sigma_K}{\sigma_C}$ and $R(C,K)= 1$. Hence, we obtain
$$ R(K,C)R(C,K)= \frac{s(K)}{s(C)} = \max\left \{\frac{s(K)}{s(C)} ,\frac{s(C)}{s(K)} \right \}.$$
\end{proof}

\begin{rem}[Comments on Theorem~\ref{theo:fejesToth}]
Let again $T\subseteq \B_2$ be a regular simplex with all its vertices on the Euclidean unit sphere,  
$\sigma_K,\sigma_C \in [1,d]$, $K:= \conv(T \cup \frac1{\sigma_K}\B_2)$, and $
C:= \conv(T \cup \frac1{\sigma_C}\B_2)$. 
With the help of Lemma~\ref{lem:convTcupB}, it is easy to verify that the pair $(K,C)$ fulfills
\eqref{eq:fejesToth} with equality for all choices of $\sigma_K,\sigma_C$, too.

On the other hand, the body $K$ in Figure~\ref{fig:SversusS0} shows that $s(K)$ cannot be replaced by $s_0(K)$ in Theorem~\ref{theo:fejesToth}.
\end{rem}

\begin{figure}[h]
\centering
\includegraphics[width=0.4\textwidth]{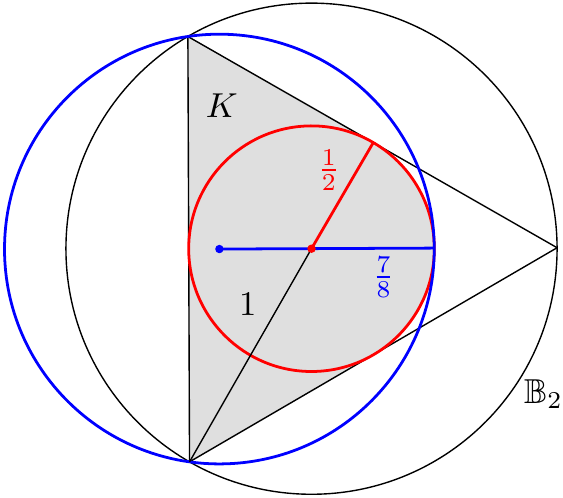}
\caption{Let $p^1, p^2, p^3 \in \S_2$, \st~$T=\conv\{p^1,p^2,p^3\}$ is a regular triangle and
  $K = \conv\left([p^1,p^2] \cup \frac12\B_2\right)$ (in gray). Then $r(K, \B_2)=\frac12$,  $R(K, \B_2)=\frac78$, and
  $s_0(K)=2$. Thus $\frac{R(K, \B_2)}{r(K,\B_2)} = \frac74 < 2 =s_0(K)$,  which shows that $s(K)$ 
  cannot be replaced by $s_0(K)$ in Theorem \ref{theo:fejesToth}. \label{fig:SversusS0}}

\end{figure}

\begin{remark}[A chain of inequalities]
Combining Theorem~\ref{theo:bohnenblust}, Corollary \ref{cor:leichtweiss}, and Theorem~\ref{theo:fejesToth}, we obtain the following chain of inequalities for $K,C \in \CC^d_0$, with $C$ symmetric. 
\begin{equation}
  \begin{split}
    2r(K,C) & \leq w(K,C) \leq (1+s(K))r(K,C) \leq r(K,C) + R(K,C) \\ & \leq \frac{1+s(K)}{s(K)}R(K,C)\leq D(K,C) \leq 2R(K,C).
  \end{split}
\label{eq:chain}
\end{equation}
\end{remark}

\paragraph{} 
With \eqref{eq:chain}, it is now immediate to confirm that in every normed space all three generalized inequalities \eqref{eq:bohnenblust}, \eqref{eq:leichtweiss}, \eqref{eq:fejesToth} are tight for any set $K$ of constant width (i.e.~ for all $K$, s.t.~$K-K=C$). 

However, since $s(K) = d$ is attained only for (fulldimensional) simplices, the equality chain 
$w(K) = (1+d)r(K) = r(K) + R(K) = \frac{1+d}{d} R(K) = D(K)$ can only hold true if there is a simplex
 $K$ of constant width, which means $r(K-K,C)=w(K,C)=D(K,C)=R(K-K,C)$ and thus
the unit ball of that space must be a central symmetrization of the simplex $K$.
The fact that, in Euclidean spaces of dimension at least 2, simplices cannot be of constant width retrospectively explains the case distinction
in \eqref{eq:JungGeneralized} and \eqref{eq:SteinhagenGeneralized}.

\paragraph{} Furthermore, the inequality 
$$ \frac{w(K,\B_2)}{R(K, \B_2)} \le 
\begin{cases}  2\sqrt{\frac{1}{d}},& \text{if } d \text{ is odd}\\
  \frac{2(d+1)}{d \sqrt{d+2}},& \text{if } d \text{ is even.}
\end{cases}$$ 
by Alexander \cite{alexander} (independently found in \cite{gritzmannLassak}), relating the width and circumradius of simplices in Euclidean space is an immediate consequence of combining \eqref{eq:SteinhagenGeneralized} and \eqref{eq:fejesToth}.
Allowing sets $K$ of arbitrary Minkowski asymmetry, we obtain
two new  inequalities for general symmetric $C$ directly from \eqref{eq:chain} and two for the euclidean case from combining
  \eqref{eq:fejesToth} with \eqref{eq:JungGeneralized} or \eqref{eq:SteinhagenGeneralized}, respectively:
\begin{corollary}[Generalized analogues to Alexander's Inequality] \label{cor:generalized_alex}
  Let $K,C \in \CC_0^d$ and $C$ be 0-symmetric. Then
  \begin{enumerate}[a)]
  \item $\displaystyle\frac{w(K,C)}{R(K,C)} \le 1 + \frac1{s(K)}$ and 
    $\displaystyle\frac{w(K,\B_2)}{R(K,\B_2)} \le \begin{cases}
      \min \left\{\frac{2\sqrt{d}}{s(K)},1+\frac1{s(K)}\right\} & \text{$d$ odd} \\
      \min \left\{\frac{2(d+1)}{s(K)\sqrt{d+2}}, 1+\frac1{s(K)} \right\} & \text{$d$ even,} \\
    \end{cases} $
  \item $\displaystyle\frac{r(K,C)}{D(K,C)} \le \frac1{s(K)+1}$ and     
    $\displaystyle\frac{r(K,\B_2)}{D(K,\B_2)} \le \min\left\{\frac{\sqrt{d}}{s(K)\sqrt{2(d+1)}},\frac{1}{s(K)+1} \right\}. $    
  \end{enumerate}
  The two inequalities are tight exactly for the examples used to show that the corresponding 
  inequalities \eqref{eq:JungGeneralized} or \eqref{eq:SteinhagenGeneralized} are tight.
\end{corollary}

\section{John's Theorem}
\label{sec:john}
\paragraph{} Finally, we cross over from containment problems under homothetics to those 
under affinities.  The most famous containment problem under affinities probably is computing ellipsoids of 
maximal volume contained in convex bodies. 
In particular the second part of Proposition~\ref{prop:john},  
which states that $\B_2$ beeing the ellipsoid of maximal volume in $K$ ensures 
that $K\subseteq d \, \B_2$, is an indispensable tool when it comes to approximations of convex bodies by 
simpler geometric objects.  We give an improved version of this part of the theorem in two ways: First,
we obtain a new lower bound in terms of the Minkowski asymmetry by Theorem \ref{theo:fejesToth}. Second, we present a simplified proof of the sharpened upper bound that is also obtained in \cite[Theorem 9]{bf-08}.

\begin{theo}[Sharpening John's Theorem]
Let $K\in\CC_0^d$ such that $\B_2$ is the ellipsoid of maximal volume enclosed in $K$. Then
$K \subseteq \rho \B_2$, where $s(K) \le \rho \le \sqrt{s_0(K)d}$.
\label{theo:strongJohn}
\end{theo}
\begin{proof}
The lower bound on $\rho$ directly follows from applying Theorem \ref{theo:fejesToth} on $K$ and the optimal ellipsoid contained in $K$ as $C$, noticing that $s(C)=1$ and therefore 
$\max\left\{\frac{s(K)}{s(C)},\frac{s(C)}{s(K)}\right\} = s(K)$.

Now, consider the upper bound on $\rho$:
If $\B_2$ is the ellipsoid of maximal volume enclosed in $K$, by John's Theorem (Proposition~\ref{prop:john}), for some $k\in \{d+1,\dots ,\frac{d+3}{2}\}$, there exist $u_1,\dots, u_k \in \bd(K) \cap \S_2$ and $\lambda_1,\dots, \lambda_k >0$ which satisfy

\begin{equation} \sum\limits_{i=1}^k \lambda_i u_i= 0 \label{eq:john1}\quad \text{ and } \quad \sum\limits_{i=1}^k \lambda_i u_i u_i^T = I \end{equation}

First, observe that, because of \eqref{eq:john1}, $\sum_{i=1}^k \lambda_i = \mathrm{trace}(I)=d$ and that 
$$-\frac{1}{s_0(K)} K\subseteq K \subseteq \{x\in \R^d: u_i^T x \leq 1 ~\forall i\in[k]\},$$ 
which means 
$u_i^T \left(-\frac{1}{s_0(K)}x\right) \leq 1$ and therefore
$-s_0(K)\leq u_i^Tx\leq 1$
for all $x \in K$ and all $i \in [k]$.
Together with $\lambda_i >0$ for $i\in [k]$ and the identities in (\ref{eq:john1}), this yields for every $x \in K$
\begin{eqnarray}
0 	& \leq  & \sum\limits_{i=1}^k \lambda_i (1-u_i^T x)(s_0(K) + u_i^Tx) \nonumber \\
	& =		& \sum\limits_{i=1}^k \lambda_i \left( s_0(K) + u_i^T x - s_0(K) u_i^T x - (u_i^Tx)^2\right) \nonumber \\
	& =		& \left(\sum\limits_{i=1}^k \lambda_i\right) s_0(K) + (1-s_0(K))\left(\sum\limits_{i=1}^k \lambda_i u_i\right)^T x - \|x\|_2^2 \nonumber \\
	& =		& d s_0(K) -\|x\|_2^2. \nonumber
\end{eqnarray}
Thus, $\|x\|_2\leq \sqrt{s_0(K)d}$.
\end{proof}

\paragraph{} Replacing the John asymmetry by the Loewner asymmetry as suggested in Remark~\ref{rem:Loewner} one can derive the same results as above for the latter one. Surely it would be even better if one could replace $s_0$ by the Minkowski asymmetry $s \le s_0$, which already was conjectured to be true in \cite{bf-08}, but seems to be more challenging.

\paragraph{} If a polytope $P \subseteq \R^d$ is given in $\CH$-presentation, it is shown in \cite{khachiyanTodd-complexityEllipsoid} that the ellipsoid of maximal volume inscribed to $P$ can be approximated to arbitrary accuracy in polynomial time. (See also \cite{toddYildirim} and the extensive list of references therein.) 
It is not known, on the other hand, whether the same is true for the minimum volume enclosing ellipsoid of $P$. In fact, it is conjectured in \cite{khachiyanTodd-complexityEllipsoid} that approximation to arbitrary accuracy of the minimum volume enclosing ellipsoid of an $\CH$-presented polytope is $\NP$-hard. 

An approximation with a multiplicative error factor of at most $(1+\varepsilon)d$, however, is readily provided by combining the algorithm mentioned above and John's Theorem. Depending on the input polytope $P$, the Sharpened inequality in Theorem~\ref{theo:strongJohn}  allows to improve this bound to $(1+\varepsilon)\sqrt{s_0(P)d}$, where the coefficient $s_0(P)$  can be computed (approximated) via Linear Programming once (an approximation of) the center of the ellipsoid of maximum volume contained in $P$ is known.
Taking into account the hardness of approximating the circumradius of an $\CH$-presented polytope even around a fixed center (cf. \cite{b-02,normmaxW1}), the improvement of the bound by the computation of $s_0(P)$ is quasi at no cost.

\paragraph{Acknowlegements.} 
We would like to thank Salvador Segura Gomis for asking the right questions and 
Bernardo González Merino and Andreas Schulz for useful pointers to relevant literature.

\bibliographystyle{plain}
\bibliography{references}

\end{document}